\documentclass[a4paper,reqno,11pt, oneside]{amsart}
\usepackage[pdftex]{graphicx,xcolor,hyperref}
\usepackage{amssymb}
\usepackage{amstext}
\usepackage{amsmath}
\usepackage{amscd}
\usepackage{amsthm}
\usepackage{amsfonts}
\usepackage{enumerate}
\usepackage{caption}
\usepackage{color}
\usepackage{here}
\usepackage{bm}
\usepackage{tikz}
\usepackage{calc}
\usepackage{multirow}
\usepackage{makecell}
\usetikzlibrary{decorations.markings}
\usetikzlibrary{positioning}
\tikzstyle{vertex}=[circle ,draw, inner sep=0pt, minimum size=6pt]

\makeatletter
  
  \@addtoreset{equation}{section}
\makeatother

\newcommand{\Ac}{\mathcal{A}}

\newcommand{\ZZ}{\mathbb{Z}}

\newcommand{\RR}{\mathbb{R}}

\newcommand{\eb}{\mathbf{e}}

\def\opn#1#2{\def#1{\operatorname{#2}}} % to make operators
\opn\conv{conv} \opn\dep{depth} \opn\Spec{Spec} \opn\cone{cone} \opn\ini{in} \opn\codeg{codeg} \opn\deg{deg}
\opn\Graph{Graph} \opn\sign{sign} \opn\Ehr{Ehr} \opn\rank{rank} \opn\type{type} \opn\reg{reg} \opn\core{core}
\opn\Hilb{Hilb} \opn\Indeg{Indeg} \opn\link{link} \opn\Tor{Tor} \opn\MNF{MNF} \opn\dep{depth} \opn\pdim{pdim}

\newtheorem{thm}{Theorem}[section]

\newtheorem{prop}[thm]{Proposition}

\newtheorem{q}[thm]{Question}

\theoremstyle{definition}

\newtheorem{ex}[thm]{Example}

\theoremstyle{remark}
\newtheorem{rem}[thm]{Remark}

\begin{document}

\title{A new class of magic positive Ehrhart polynomials of reflexive polytopes}
\author{Masato Konoike}

\address[M. Konoike]{Department of Pure and Applied Mathematics, 
Graduate School of Information Science and Technology, 
Osaka University, Japan}
\email{kounoike-m@ist.osaka-u.ac.jp}


\maketitle

\begin{abstract} 
The magic positivity of Ehrhart polynomials is a useful tool for proving the real-rootedness of the $h^\ast$-polynomials. 
In this paper, we provide a new class of reflexive polytopes whose Ehrhart polynomials are magic positive. 
First, we prove that the Ehrhart polynomials of Stasheff polytopes are magic positive. 
Second, we provide a partial proof of the magic positivity of the Ehrhart polynomials of the dual polytopes of the symmetric edge polytopes of cycles.
\end{abstract}

\section{Introduction}
A \textit{lattice polytope} is the convex hull of finitely many elements in a lattice contained in $\mathbb{R}^d$, typically $\mathbb{Z}^d$. A lattice polytope $P$ is called \textit{reflexive} if
$$P^\ast := \{ y \in \RR^d \mid \langle x, y\rangle \leq 1 \text{ for any } x \in P\}$$
is also a lattice polytope, where $\langle \cdot, \cdot \rangle$ denotes the usual inner product of $\RR^d$. 

By a theorem of Ehrhart \cite{E}, $|nP\cap \mathbb{Z}^d|$ is given by a polynomial $E_P(n)$ of degree $\dim P$ in $n$ for all integers $n > 0$. The polynomial $E_P(n)$ is called the \textit{Ehrhart polynomial} of $P$. The \textit{$h^\ast$-polynomial} $h^\ast _P (t)=h_0^\ast +h_1^\ast t+\cdots +h_d^\ast t^d$ of a $d$-dimensional lattice polytope $P$ encodes the Ehrhart polynomial in a particular basis consisting of binomial coefficients:
$$E_P (n)=h_0^\ast \binom{n+d}{d} +h_1^\ast \binom{n+d-1}{d} +\cdots + h_d^\ast \binom{n}{d}.$$
A fundamental theorem by Stanley \cite{RS80} states that the coefficients of the $h^\ast$-polynomial are always nonnegative integers. It was proved by Hibi~\cite{H} that a $d$-dimensional lattice polytope $P$ is reflexive if and only if its $h^\ast$-polynomial is palindromic and has degree $d$, that is, $h^\ast _P (t)=t^dh^\ast _P \left(\frac{1}{t}\right)$.

Consider the polynomial $f(n)$ expressed in a different basis:
$$f(n) = \sum_{i=0}^d a_i n^i (1+n)^{d-i}.$$
If $a_0,\ldots,a_d\geq 0$, then $f(n)$ is said to be \textit{magic positive}. The term `magic positive` was introduced by Ferroni and Higashitani \cite{FH}. A common area of study is whether $h^\ast$-polynomials are real-rooted. See Section 2 for the real-rootedness of polynomials. The motivation for studying the magic positivity comes from the following theorem.

\begin{thm}[{\cite[Theorem 4.19]{FH}}]
Let $P$ be a lattice polytope of dimension $d$. Consider the Ehrhart polynomial written in a different basis:
$$E_P(n) = \sum_{i=0}^d a_i n^i (1+n)^{d-i}.$$
If $a_0,\ldots,a_d\geq 0$, then we have simultaneously that $E_P(n)$ has non-negative coefficients and $h^*_P(n)$ is real-rooted.
\end{thm}

Therefore, the following question naturally arises.
\begin{q}
When is the Ehrhart polynomial of lattice polytopes magic positive?
\end{q}
It is known that the Ehrhart polynomials of zonotopes \cite{BJM} and Pitman-Stanley polytopes \cite{FAM} are magic positive.

The CL-ness of cross polytopes, dual of the Stasheff polytopes, root polytopes of type A, and root polytopes of type C has been studied in \cite{HKM}. Here, a reflexive polytope $P$ is said to be a \textit{CL-polytope} if all of the complex roots of the Ehrhart polynomial $E_P(n)$ lie on the line $\mathrm{Re}(z) = -\frac{1}{2}$, where $\mathrm{Re}(z)$ denotes the real part of $z$. Since the dual of these three polytopes, except for the dual of the Stasheff polytopes are zonotopes \cite{CM}, their Ehrhart polynomials are magic positive. However, since Stasheff polytopes are not zonotopes, it is not immediately clear whether their Ehrhart polynomials are magic positive. The answer to this question is presented in the following first main theorem. 
\begin{thm}\label{main1}
The Ehrhart polynomials of Stasheff polytopes are magic positive.
\end{thm}

We also study the magic positivity of the Ehrhart polynomials of the dual of symmetric edge polytopes. The result of this study is presented in the following second main theorem.
\begin{thm}\label{main2}
We transform $E_{P_{C_{d+1}}^\ast}(n)$ into the form $E_{P_{C_{d+1}}^\ast}(n) = \sum _{j=0}^da_jn^j(1+n)^{d-j}$. Then, the coefficients $a_i$ and $ a_{d-i}$ for $i = 0, 1, 2$ are positive. Here, $P_{C_{d+1}}$ represents symmetric edge polytopes of cycles.
\end{thm}

This paper is organized as follows. 
In Section~\ref{sec:pre}, we prepare the materials for the proofs of the main results. 
In Section~\ref{sec:main1}, we provide a proof of Theorem~\ref{main1}. 
In Section~\ref{sec:main2}, we investigate the magic positivity of the Ehrhart polynomials of the dual of symmetric edge polytopes and provide a proof of Theorem~\ref{main2}. 

\subsection*{Acknowledgements}
The author would like to thank Akihiro Higashitani for his helpful comments and advice on improving this paper.

\medskip

\section{Preliminaries}\label{sec:pre}

In this section, we recall several materials used in this paper. 

\subsection{Polynomial real-rooted}
A polynomial $f=\sum_{i=0} ^d a_i t^i$ of degree $d$ with real coefficients is said to be \textit{real-rooted}, if all its roots are real. If all the coefficients of a real-rooted polynomial are nonnegative, or equivalently, if all its roots are nonpositive, then $a_i^2\geq a_{i-1}a_{i+1}$ for all $i$~\cite{RS89}. A sequence $a_i$ of coefficients satisfying these inequalities is called \textit{log-concave}. An immediate consequence is that the nonnegative, log-concave sequence is \textit{unimodal}, meaning $a_0\leq a_1\leq \cdots \leq a_k \geq \cdots \geq a_d$ for some $k$.

\subsection{Stasheff polytope}
Let $\mathrm{St}_d$ be the convex hull of $\{ \pm \eb_i : 1 \leq i \leq d\} \cup \{\eb_i+\cdots+\eb_j : 1 \leq i < j \leq d\}$. 
Then $\mathrm{St}_d$ is a reflexive polytope of dimension $d$. 
This polytope is the dual polytope of the so-called  \textit{Stasheff polytope} (\textit{associahedron}). 
For more detailed information, see, e.g., \cite{FZ}.

\subsection{Symmetric edge polytope}
Let $G$ be a finite simple graph on the vertex set $[d]$ and the edge set $E(G)$. The \textit{symmetric edge polytope} $P_G \subset \RR^d$ is the convex hull of the set
$$\{\pm(e_v-e_w) \in \RR^d \mid vw \in E(G)\}$$
Here, the vectors $e_v$ are elements that form a lattice basis of $\ZZ^d$. It is known that \cite{OH14} $P_G$ is reflexive. For more context on symmetric edge polytopes, see e.g. \cite{HJM, MHNOH}.

\subsection{Magic positive}
We use the following proposition to prove Theorem~\ref{main2}.
\begin{prop}[{\cite[Proposition 4.11]{BJM}}]
Let $P$ be a $d$-dimensional lattice polytope and let 
$$E_{P}(n) = \sum _{i=0}^da_in^i(1+n)^{d-i}$$
be its Ehrhart polynomial. Then $P$ is reflexive if and only if $a_j=a_{d-j}$.
\end{prop}

We explain why the dual of cross polytopes, and root polytopes of type A and type C have Ehrhart polynomials that are magic positive.

\begin{ex}[Dual of the Cross Polytope]\label{cross}
Let $\mathrm{Cr}_d$ be the convex hull of $\{\pm \eb_i \mid 1 \leq i \leq d\}$. 
Then $\mathrm{Cr}_d$ is a reflexive polytope of dimension $d$, called the \textit{cross polytope}.
Since dual polytopes of $\mathrm{Cr}_d$ and $[-1,1]^d$ are equal, the Ehrhart polynomial of dual polytopes of $\mathrm{Cr}_d$ can be computed as follows: 
$$E_{\mathrm{Cr}_d^\ast}(n)=(2n+1)^d=\sum_{i=0}^{d}\binom{d}{i}(n+1)^{d-i}n^i.$$
Therefore $E_{\mathrm{Cr}_d^\ast}(n)$ is magic positive.
\end{ex}

\begin{ex}[Dual of the Root Polytope of Type A]\label{typeA}
Let ${\bf A}_d$ be the convex hull $\{\pm \eb_i \mid 1 \leq i \leq d\} \cup \{ \pm (\eb_i+\cdots+\eb_j) \mid 1 \leq i < j \leq d\}$. 
Then ${\bf A}_d$ is a reflexive polytope of dimension $d$, called the \textit{root polytope of type A}. 
The Ehrhart polynomial of dual polytope of ${\bf A}_d$ is calculated in \cite[Lemma 5.3]{HKM} as follows:
$$E_{{\bf A}_d^\ast}(n)=\sum_{k=0}^d \binom{d+1}{k}n^k=\sum_{i=0}^d (n+1)^{d-i}n^i.$$
Therefore $E_{{\bf A}_d^\ast}(n)$ is magic positive.
\end{ex}

\begin{ex}[Dual of the Root Polytope of Type C]\label{typeC}
Let ${\bf C}_d$ be the convex hull of $\{\pm \eb_i \mid 1 \leq i \leq d \} \cup \{ \pm (\eb_i+\cdots+\eb_{j-1}) \mid 1 \leq i < j \leq d \} \cup \{\pm(2\eb_i+\cdots+2\eb_{d-1}+\eb_d) \mid 1 \leq i \leq d-1\}$. 
Then ${\bf C}_d$ is a reflexive polytope of dimension $d$, called the \textit{root polytope of type C}. 
The Ehrhart polynomial of dual polytope of ${\bf C}_d$ is calculated in \cite[Theorem 1.1]{HY} as follows: 
$$E_{{\bf C}_d^\ast}(n)=(n+1)^d + n^d.$$
From this indication, it is clear that $E_{{\bf C}_d^\ast}(n)$ is magic positive.
\end{ex}

\section{Proof of Theorem~\ref{main1}}\label{sec:main1}
In this section, we give a proof of Theorem~\ref{main1}. 
For the proof of Theorem~\ref{main1}, we use the following theorem.
\begin{thm}
For $d \geq 2, E_{\mathrm{St}_d}(n)$ satisfies the following recurrence:
\begin{align}
  E_{\mathrm{St}_d^\ast}(n) = (2n+1)E_{\mathrm{St}_{d-1}^\ast}(n) - \frac{1}{2}n(n+1)E_{\mathrm{St}_{d-2}^\ast}(n).
\end{align}
\end{thm}
\begin{proof}
First, we define some sets. By the definition of the Stasheff polytope,
\[
\mathrm{St}_d^\ast = \{(x_1, \ldots, x_d) \in \RR^d \mid -1 \leq x_i \leq 1, x_j + \cdots + x_k \leq 1 (1 \leq i \leq d, 1 \leq j < k \leq d)\},
\]
\[
\mathrm{St}_{d-1}^\ast \times [-1, 1] = \{(x_1, \ldots, x_d) \in \RR^d \mid -1 \leq x_i \leq 1, x_j + \cdots + x_k \leq 1 (1 \leq i \leq d, 1 \leq j < k \leq d-1)\}.
\]
Let
$$\Ac = (\mathrm{St}_{d-1}^\ast \times [-1, 1]) \backslash \mathrm{St}_d^\ast.$$
For integers $m$ with $-n \leq m \leq n$, let
\[
A_m = \left(\bigcup_{i = 1}^{d - 2}(n\mathrm{St}_{d-2}^\ast \cap \{(x_1, \ldots, x_{d-2}) \in \RR^{d-2} \mid x_i + \cdots + x_{d-2} = m\})\right) \backslash \bigcup_{i = m+1}^{n}A_i,
\]
\[
\Delta_m = \{(x_{d-1}, x_d) \in \RR^2 \mid x_{d-1} \leq n-\max\{m, 0\}, x_d \leq n, x_{d-1} + x_d > n -\max\{m, 0\}\}.
\]
The following equation can be easily verified:
\begin{align}
  E_{\mathrm{St}_{d-1}^\ast \times [-1, 1]}(n) = (2n+1)E_{\mathrm{St}_{d-1}^\ast}(n).
\end{align}
Since $\mathrm{St}_d^\ast \subset \mathrm{St}_{d-1}^\ast \times [-1, 1]$, by (3.2)
\begin{align}
  \left|n\Ac \cap \ZZ^d\right| = (2n+1)E_{\mathrm{St}_{d-1}^\ast}(n) - E_{\mathrm{St}_d^\ast}(n).
\end{align}
Since $\Delta_m$ and $\{(x,y) \in \RR^2 \mid x \geq 0, y \geq 0, x+y < n\}$ are unimodularly equivalent, we have the following equation:
\begin{align}
  |\Delta_m \cap \ZZ^2| = \frac{1}{2}n(n+1).
\end{align}

\smallskip

\noindent
\underline{The first step}: We prove the following equality:
\begin{align}
  \bigsqcup_{-n \leq m \leq n}A_m \cap \ZZ^{d-2} = n\mathrm{St}_{d-2}^\ast \cap \ZZ^{d-2}.
\end{align}
Since $\bigsqcup_{-n \leq m \leq n}A_m \cap \ZZ^{d-2} \subset n\mathrm{St}_{d-2}^\ast \cap \ZZ^{d-2}$ is clear, we will prove $\bigsqcup_{-n \leq m \leq n}A_m \cap \ZZ^{d-2} \supset n\mathrm{St}_{d-2}^\ast \cap \ZZ^{d-2}$. For $x = (x_1, \ldots, x_{d-2}) \in n\mathrm{St}_{d-2}^\ast \cap \ZZ^{d-2}$, we define the value $p$ as follows:
$$p = \max{\{x_i + \cdots + x_{d-2} \mid 1 \leq i \leq d-2\}}.$$
Then, since $-n \leq p \leq n$ holds, and $x \in A_p \cap \ZZ^{d-2}$ follows from the definition of $A_m$. Therefore, we get (3.5).

From (3.4) and (3.5), we have the following equation:
\begin{align}
  \left|\bigsqcup_{-n \leq m \leq n}(A_m \times \Delta_m) \cap \ZZ^d \right| = \frac{1}{2}n(n+1)E_{\mathrm{St}_{d-2}^\ast}(n).
\end{align}

\smallskip

\noindent
\underline{The second step}: Next, we prove the following equality:
\begin{align}
  n\Ac \cap \ZZ^d = \bigsqcup_{-n \leq m \leq n}(A_m \times \Delta_m) \cap \ZZ^d.
\end{align}
Let $x = (x_1, \ldots, x_d) \in n\Ac \cap \ZZ^d$. Then, $x_d \leq n$ holds, and since $(x_1, \ldots, x_{d-2}) \in n\mathrm{St}_{d-2}^\ast \cap \ZZ^{d-2}$, that is, $(x_1, \ldots, x_{d-2}) \in \bigsqcup_{-n \leq m \leq n}A_m \cap \ZZ^{d-2}$ holds by (3.5), there exists an integer $-n \leq k \leq n$ such that $(x_1, \ldots, x_{d-2}) \in A_k$. 

Now, suppose $x_{d-1} > n-\max\{k, 0\}$. From the first step, since there exists $1 \leq i' \leq d-2$ such that $x_{i'} + \cdots + x_{d-2} = \max\{k, 0\}$, we have $x_{i'} + \cdots + x_{d-1} > n$, implying $x \notin n(\mathrm{St}_{d-1}^\ast \times [-1, 1]) \cap \ZZ^d$, which contradicts $x \in n\Ac$. Therefore, we conclude that $x_{d-1} \leq n-\max\{k, 0\}$. 

Next, suppose $x_{d-1}+x_d \leq n-\max\{k, 0\}$. From $(x_1, \ldots, x_{d-2}) \in n\mathrm{St}_{d-2}^\ast \cap \ZZ^{d-2}$, we know that $x_{i} + \cdots + x_{d-2} \leq \max\{k, 0\}$ for all $1 \leq i \leq d-2$. Consequently, $x_{i} + \cdots + x_d \leq n$ for all $1 \leq i \leq d$, implying $x \in n\mathrm{St}_d^\ast\cap \ZZ^d$, which contradicts $x \in n\mathrm{St}_d^\ast \cap \ZZ^d$. Thus, we deduce $x_{d-1}+x_d > n-\max\{k, 0\}$.
Therefore, we have $x \in \bigsqcup_{-n \leq m \leq n}(A_m \times \Delta_m) \cap \ZZ^d$.

Now, let $x = (x_1, \ldots, x_d) \in \bigsqcup_{-n \leq m \leq n}(A_m \times \Delta_m) \cap \ZZ^d$. Then there exists $-n \leq k \leq n$ such that $(x_1, \ldots, x_d) \in A_k \times \Delta_k$. Since $x_i + \cdots + x_{d-2} \leq \max\{k, 0\}$ holds for all $1 \leq i \leq d-2$ by $(x_1, \ldots, x_{d-2}) \in A_k$, we have $x_i + \cdots + x_{d-1} \leq n$. 
Additionally, since there exists $1 \leq i' \leq d-2$ such that $x_{i'} + \cdots + x_{d-2} = \max\{k, 0\}$ and $x_{d-1} + x_d > n-\max\{k, 0\}$ by $(x_{d-1}, x_d) \in \Delta_k$, we conclude $x_{i'} + \cdots + x_d > n$. Therefore, $x \in n\Ac \cap \ZZ^d$.

\smallskip

\noindent
\underline{The third step}: From (3.3) and (3.7), we have the following equation:
$$\left| \bigsqcup_{-n \leq m \leq n}(A_m \times \Delta_m) \cap \ZZ^d \right|= (2n+1)E_{\mathrm{St}_{d-1}^\ast}(n) - E_{\mathrm{St}_d^\ast}(n).$$
From this equation and (3.6), we get (3.1).
\end{proof}

By using Theorem 3.1, we can prove Theorem 1.3 without explicitly determining $E_{\mathrm{St}_d^\ast}(n)$.
\begin{proof}[Proof of Theorem~\ref{main1}]
Transform $(3.1)$ into the following form$:$
$$E_{\mathrm{St}_d^\ast}(n) - \frac{1}{2}nE_{\mathrm{St}_{d-1}^\ast}(n) = \frac{1}{2}nE_{\mathrm{St}_{d-1}^\ast}(n) + (n+1)\left(E_{\mathrm{St}_{d-1}^\ast}(n) - \frac{1}{2}nE_{\mathrm{St}_{d-2}^\ast}(n)\right).$$
By the recurrence relation, we observe that if both $E_{\mathrm{St}_{d-1}^\ast}(n) - \frac{1}{2}nE_{\mathrm{St}_{d-2}^\ast}(n)$ and $\frac{1}{2}nE_{\mathrm{St}_{d-1}^\ast}(n)$ are magic positive, so are $E_{\mathrm{St}_d^\ast}(n) - \frac{1}{2}nE_{\mathrm{St}_{d-1}^\ast}(n)$ and $E_{\mathrm{St}_d^\ast}(n)$. In other words, it suffices to show that $E_{\mathrm{St}_1^\ast}(n) $ and $E_{\mathrm{St}_1^\ast}(n) - \frac{1}{2}nE_{\mathrm{St}_0^\ast}(n)$ are magic positive. 
In fact, we have $E_{\mathrm{St}_0^\ast}(n) = 1, E_{\mathrm{St}_1^\ast}(n) = 2n+1 = (n+1) + n$ and $E_{\mathrm{St}_1^\ast}(n) - \frac{1}{2}nE_{\mathrm{St}_0^\ast}(n) = (n+1) + \frac{1}{2}n$, as required. Therefore, $E_{\mathrm{St}_d^\ast}(n)$ is magic positive.
\end{proof}

The following table shows examples of the Ehrhart polynomials of the Stasheff polytope in dimensions 2, 3, 4, and 5, where it can be confirmed that they are magic positive.

\begin{table}[h]
  \centering
  \renewcommand{\arraystretch}{1.3}
  \begin{tabular}{|l|}
     \hline
      $\begin{array}{rcl}
      E_{\mathrm{St}_2}(n) & = & \frac{7}{2}n^2 + \frac{7}{2}n + 1 \\
        & = & (n+1)^2 + \frac{3}{2}(n+1)n + n^2
      \end{array}$ \\ \hline
      $\begin{array}{rcl}
      E_{\mathrm{St}_3}(n) & = & 6n^3 + 9n^2 + 5n + 1 \\
        & = & (n+1)^3 + 2(n+1)^2n + 2(n+1)n^2 + n^3
      \end{array}$ \\ \hline
      $\begin{array}{rcl}
      E_{\mathrm{St}_4}(n) & = & \frac{41}{4}n^4 + \frac{41}{2}n^3 + \frac{67}{4}n^2 + \frac{13}{2}n + 1 \\
        & = & (n+1)^4 + \frac{5}{2}(n+1)^3n + \frac{13}{4}(n+1)^2n^2 + \frac{5}{2}(n+1)n^3 + n^4
      \end{array}$ \\ \hline
      $\begin{array}{rcl}
      E_{\mathrm{St}_5}(n) & = & \frac{35}{2}n^5 + \frac{175}{4}n^4 + 47n^3 + \frac{107}{4}n^2 + 8n + 1 \\
        & = & (n+1)^5 + 3(n+1)^4n + \frac{19}{4}(n+1)^3n^2 + \frac{19}{4}(n+1)^2n^3 + 3(n+1)n^4 + n^5
      \end{array}$ \\ \hline
  \end{tabular}
\end{table}

\begin{rem}
From \cite[Theorem 2.2]{RS91}, the coefficients of the Ehrhart polynomial of a zonotope are all integers. On the other hand, since the coefficients of the Ehrhart polynomial of a Stasheff polytope are rational numbers, the Stasheff polytope is not zonotope.
\end{rem}

\section{Proof of Theorem~\ref{main2}}\label{sec:main2}
In this section, we provide the background that led us to consider Theorem~\ref{main2}, followed by a proof of Theorem~\ref{main2}. We begin by examining the cases of tree $T_d$ and complete graph $K_d$.
\begin{prop}
The Ehrhart polynomials of the dual of the symmetric edge polytopes of trees and complete graphs are magic positive.
\end{prop}
\begin{proof}
From {\cite[Proposition 4.6]{MHNOH}}, $P_{T_{d+1}}$ and $\mathrm{Cr}_d$ are unimodularly equivalent. Therefore, by Example 3.1, the Ehrhart polynomials of $P_{T_{d+1}}^\ast$ are magic positive. 

Moreover, $P_{K_{d+1}}$ and ${\bf A}_d$ are unimodularly equivalent. Therefore, by Example 3.2, the Ehrhart polynomials of $P_{K_{d+1}}^\ast$ are magic positive.
\end{proof}

For any connected graph $G$ with $d$ vertices, the following inclusions hold:
$$P_{K_d}^\ast \subset P_{G}^\ast \subset P_{T_d}^\ast.$$
From these relations and Proposition 4.1, the following question naturally arises.
\begin{q}
Is the Ehrhart polynomial of the dual polytopes of symmetric edge polytopes of any connected graph magic positive?
\end{q}
We expected this question to hold, but through computational experiments, we found several counterexamples. First, we present a counterexample in the case of a complete bipartite graph. For instance, $E_{P_{K_{3,7}}^\ast}$ is not magic positive. Specifically, $E_{P_{K_{3,7}}^\ast}(n)$ is computed as follows:
\begin{align*}
E_{P_{K_{3,7}}^\ast}(n) &= \frac{128}{3}n^9 + 192n^8 + 448n^7 + 672n^6 + \frac{3528}{5}n^5 + 532n^4 + \frac{820}{3}n^3 + 86n^2 + \frac{72}{5}n + 1 \\
&=(n+1)^9 + \frac{27}{5}(n+1)^8n + \frac{34}{5}(n+1)^7n^2 -\frac{142}{15}(n+1)^6n^3 + \frac{88}{5}(n+1)^5n^4\\
& + \frac{88}{5}(n+1)^4n^5 -\frac{142}{15}(n+1)^3n^6 + \frac{34}{5}(n+1)^2n^7 + \frac{27}{5}(n+1)n^8 + n^9.
\end{align*}

The following table marks $\circ$ when $E_{P_{K_{m,n}}^\ast}(n)$ is magic positive and $\times$ when it is not.

\begin{table}[H]
  \centering
  \begin{tabular}{|c|c|c|c|c|c|c|c|c|c|c|} 
      \hline
      & 2 & 3 & 4 & 5 & 6 & 7 & 8 & 9 & 10 & 11\\ 
      \hline
      2 & $\circ$ & $\circ$ & $\circ$ & $\circ$ & $\circ$ & $\circ$ & $\circ$ & $\times$ & $\times$ & $\times$ \\ 
      \hline
      3 & $\circ$ & $\circ$ & $\circ$ & $\circ$ & $\circ$ & $\times$ & $\times$ & $\times$ & $\times$  &\\ 
      \hline
      4 & $\circ$ & $\circ$ & $\circ$ & $\circ$ & $\times$ & $\times$ & $\times$ & $\times$ &  & \\ 
      \hline
      5 & $\circ$ & $\circ$ & $\circ$ & $\times$ & $\times$ & $\times$ & $\times$ &  &  & \\ 
      \hline
      6 & $\circ$ & $\circ$ & $\times$ & $\times$ & $\times$ & $\times$ &  &  &  & \\ 
      \hline
      7 & $\circ$ & $\times$ & $\times$ & $\times$ & $\times$ &  &  &  &  & \\ 
      \hline
      8 & $\circ$ & $\times$ & $\times$ & $\times$ &  &  &  &  &  & \\ 
      \hline
      9 & $\times$ & $\times$ & $\times$ &  &  &  &  &  &  & \\ 
      \hline
      10 & $\times$ & $\times$ &  &  &  &  &  &  &  &\\ 
      \hline
      11 & $\times$ &  &  &  &  &  &  &  &  &\\ 
      \hline
   \end{tabular}
\end{table}

Except for the case of $K_{2,8}$, when the number of vertices in a graph exceeds 10, the Ehrhart polynomial of the dual of the symmetric edge polytopes tends not to be magic positive. However, the $h^\ast$-polynomials of the polytopes with Ehrhart polynomials that are not magic positive, as shown in the table above, are real rooted.

Second, we provide a counterexample for the case of removing an edge from a complete graph. For example, $E_{P_{K_{10} \backslash \{e\}}^\ast}(n)$ is not magic positive where $e$ is an edge of $K_{10}$. $E_{P_{K_{10} \backslash \{e\}}^\ast}(n)$ is calculated as follows:
\begin{align*}
E_{P_{K_{10} \backslash \{e\}}^\ast}(n) &= \frac{92}{9}n^9 + 46n^8 + \frac{364}{3}n^7 + 210n^6 + \frac{3766}{15}n^5 + 210n^4 + \frac{1084}{9}n^3 + 45n^2 + \frac{149}{15}n + 1 \\
&= (n+1)^9 + \frac{14}{15}(n+1)^8n + \frac{23}{15}(n+1)^7n^2 -\frac{19}{45}(n+1)^6n^3 + \frac{31}{15}(n+1)^5n^4 \\
&+ \frac{31}{15}(n+1)^4n^5 -\frac{19}{45}(n+1)^3n^6 + \frac{23}{15}(n+1)^2n^7 + \frac{14}{15}(n+1)n^8 + n^9.
\end{align*}

The counterexample shown at the beginning demonstrates that there exists an Ehrhart polynomial of the dual of the symmetric edge polytope of a graph close to a complete graph that is not magic positive. So, what happens to the magic positivity of the Ehrhart polynomial of the dual of the symmetric edge polytope for graphs that are close to trees.

\begin{q}
Is the Ehrhart polynomials of dual polytopes of symmetric edge polytopes of cycle magic positive?
\end{q}
We computed the Ehrhart polynomial of the dual of the symmetric edge polytopes of cycles and partially resolved its magic positivity.
\begin{prop}
Let $C_d$ denote a cycle of length $d \geq 2$. Then we have
\begin{align}
  E_{P_{C_{d+1}}^\ast} = \sum_{i=0}^{\lfloor \frac{d}{2} \rfloor} (-1)^i\binom{(d+1-2i)n + (d-i)}{d}\binom{d+1}{i}.
\end{align}
\end{prop}
\begin{proof}
By the definition of symmetric edge polytope, $P_{C_{d+1}}$ is the convex hull of $\{\pm(e_i-e_{i+1}) \mid 1 \leq i \leq d-1\} \cup \{\pm(e_d-e_1)\}$. Since $P_{C_{d+1}}$ and the convex hull of $\{\pm e_i \mid 1 \leq i \leq d-1\} \cup \{\pm (e_1 + \cdots + e_d)\}$ are unimodular equivalent, we obtain the following equality:
$$P_{C_{d+1}}^\ast = \{(x_1, \ldots, x_d) \in \RR^d \mid -1 \leq x_i \leq 1, -1 \leq x_1 + \cdots + x_d \leq 1 (1 \leq i \leq d)\}.$$
Using \cite[Proposition 2.2 and Theorem 2.5]{FM}, we obtain 
$$E_{P_{C_{d+1}}^\ast}(n) = \sum_{i=0}^{d}(-1)^i\sum_{v=0}^{d}\binom{(d+1-v)n+(d-i)}{d}\rho_{\mathbf{c}, i}(v)$$
where $\mathbf{c} = (2, \ldots, 2) \in \ZZ_{\geq 0}^{d+1}$and
$$\rho_{\mathbf{c},i}(v) := \#\left\{ I\in \binom{[d+1]}{i} : \sum_{j \in I} c_j = v\right\}=\begin{cases}
\binom{d+1}{i} &\text{if $v = 2i$}, \\
0 &\text{otherwise}. 
\end{cases}
$$
Therefore, we have (4.1).
\end{proof}

By using Proposition 4.4, we can prove Theorem 1.4.
\begin{proof}[Proof of Theorem~\ref{main2}]
Since the symmetric edge polytope is reflexive, from Proposition 2.4, it suffices to show that $a_0, a_1, a_2$ are positive. For $0 \leq i \leq \lfloor \frac{d}{2} \rfloor$, let
$$
B _i =
\begin{pmatrix}
  1-i & 2-i & \cdots & d-1-i & d-i\\
  d-i & d-1-i & \cdots & 2-i & 1-i
\end{pmatrix}
$$
be the $2 \times d$-matrix and for $I \subset [d] = \{1,2,\ldots,d\}$, let
$$B_i^I = (-1)^i\prod_{j \in I}b_{i2j}\prod_{j \in [d] \backslash I}b_{i1j}\binom{d+1}{i}, C_{I} = \sum_{i=1}^{\lfloor \frac{d}{2} \rfloor}B_i^I$$
where $b_{ijk}$ denotes the $(j,k)$-entry of the matrix $B_i$.
Since
\[
\binom{(d+1-2i)n + (d-i)}{d} = \frac{1}{d!}\prod_{j=1}^{d}((j-i)n + (d+1-j-i)(n+1)),
\]
we have $a_i = \sum_{I \in \binom{[d]}{i}}C_I$. From this equation, it suffices to show that $C_I \geq 0$ for all $I \in \binom{[d]}{i}$. Note that $B_0^I > 0$, and for all $i \in \{1, \ldots, \lfloor \frac{d}{2} \rfloor\}$, if $i \notin I$ or $d+1-i \in I$, then $B_i^I = 0$.

\noindent
Obviously $a_0 = 1$, so we prove that $a_1$ is positive. For $1 \leq i \leq d$, let $I = \{i\}$. 

\noindent
If $\lfloor \frac{d}{2} \rfloor < i \leq d$, then $B_j^{\{i\}} = 0$ for all $1 \leq j \leq \lfloor \frac{d}{2} \rfloor$. Therefore $C_I = B_0^I > 0$. 

\noindent
If $1 \leq i \leq \lfloor \frac{d}{2} \rfloor$, then
\begin{align*}
C_{I} &= B_0^I + B_{i}^I \\
&= \frac{d+1-i}{i} - \frac{1}{d!}\binom{d+1}{i}(i-1)!(d-i)!(d+1-2i) \\
&= \frac{d+1-i}{i} - (d+1)\left(\frac{1}{i} - \frac{1}{d+1-i}\right) \\
&= \frac{d+1}{d+1-i} - 1 = \frac{i}{d+1-i} > 0.
\end{align*}
Therefore $a_1$ is positive.

Next we prove that $a_2$ is positive in three cases. For $1 \leq i < j \leq d$, let $I = \{i, j\}$.
\smallskip

\noindent
\underline{The first case}: If $\lfloor \frac{d}{2} \rfloor \leq i < j \leq d$ for all $1 \leq k \leq \lfloor \frac{d}{2} \rfloor$. Therefore $C_I = B_0^I > 0$

\smallskip

\noindent
\underline{The second case}: If $1 \leq i < \lfloor \frac{d}{2} \rfloor < j \leq d$, then $C_I = \frac{d+1-j}{j}B_0^{\{i\}} + \frac{d+1-j-i}{j-i}B_i^{\{i\}}$. Since
$$\frac{d+1-j}{j} - \frac{d+1-j-i}{j-i} = \frac{i(2j-d-1)}{j(j-i)} \geq 0,$$
$B_0^{\{i\}} + B_i^{\{i\}} > 0$ and $B_0^I > 0$, we obtain $C_I > 0$
\smallskip

\noindent
\underline{The third case}: If $1 \leq i < j \leq \lfloor \frac{d}{2} \rfloor$, then
\begin{align*}
C_I &= B_0^I + B_i^I + B_j^I \\
&= \frac{(d+1-i)(d+1-j)}{ij} - \frac{d+1}{(d+1-i)i}\frac{d+1-i-j}{j-i}(d+1-2i) \\
&- \frac{d+1}{(d+1-j)j}\frac{d+1-i-j}{i-j}(d+1-2j) \\
&= \frac{(d+1-i)(d+1-j)}{ij} - \frac{(d+1)(d+1-i-j)}{j-i}\left(\frac{d+1-2i}{(d+1-i)i} + \frac{d+1-2j}{(d+1-j)j}\right) \\
&= 1 + \frac{(d+1)(d+1-i-j)}{ij} - (d+1)(d+1-i-j)\left(\frac{1}{ij} + \frac{1}{(d+1-i)(d+1-j)}\right) \\
&= 1 - \frac{(d+1)(d+1-i-j)}{(d+1-i)(d+1-j)} = \frac{ij}{(d+1-i)(d+1-j)}.
\end{align*}
Since $ij \geq 2$ and $i, j < d + 1$, $C_I > 0$. Therefore $a_2$ is positive.
\end{proof}

\begin{rem}
This method can't be used for coefficient of $(n+1)^{d-3}n^3$. For example, if $d=6, I = \{1, 2, 4\}$, then
$$C_I = B_0^I + B_1^I+ B_2^I = (6\times5\times3)^2 - 7\times(5\times4\times2)^2 + 21\times(4\times3)^2 = -76$$
\end{rem}

\begin{rem}
By computational experiments, $E_{C_{d+1}^\ast}$ is magic positive for all $d \leq 500$. 
\end{rem}

\medskip

\bibliographystyle{plain} 
\bibliography{ref}
\end{document}